\documentclass[11pt,letter]{amsart}
\usepackage{amssymb}
\usepackage{amsfonts}
\usepackage{amsmath}
\usepackage{graphicx}
\usepackage{color}
\usepackage{amsmath}
\usepackage{mathrsfs}
\usepackage{stmaryrd}
\usepackage{epsfig,color}
\usepackage{blindtext}
\usepackage{enumerate}
\usepackage{hyperref}
\usepackage{url}
\usepackage{bbm}
\usepackage{filecontents}
\usepackage{nicefrac,mathtools}
\usepackage{bm}
\DeclareGraphicsExtensions{.pdf,.jpeg,.png}
\usepackage{epstopdf}
\usepackage{cancel} %for editing
\usepackage[normalem]{ulem} %for editing
\usepackage{mdwlist}
\usepackage{verbatim}		%for comment
\usepackage{enumitem} % for alph* enumerate item

\usepackage{color}
\usepackage[msc-links, lite]{amsrefs}
\usepackage{geometry}
\newtheorem{theorem}{Theorem}[section]

\newtheorem{proposition}[theorem]{Proposition}
\newtheorem{lemma}[theorem]{Lemma}
\newtheorem{corollary}[theorem]{Corollary}

\newtheorem{claim}[]{Claim}

\theoremstyle{definition}

\theoremstyle{remark}
\newtheorem{remark}[theorem]{Remark}

\numberwithin{equation}{section}

\newcommand{\mf}{\mathbf}
\newcommand{\mb}{\mathbb}
\newcommand{\mc}{\mathcal}

\newcommand{\oli}{\overline}

\newcommand{\wti}{\widetilde}
\newcommand{\Int}{\mathrm{Int}}

\newcommand{\dist}{\operatorname{dist}}

\DeclareMathOperator{\Index}{index}

\DeclareMathOperator{\Ric}{Ric}

\DeclareMathOperator{\spt}{spt}

\title{Conformal upper bounds for the volume spectrum}

\date{\today}

\author{Zhichao Wang}
\address{University of Toronto}
\email{zhichao.wang@utoronto.ca}

\begin{document}

\begin{abstract}
In this paper, we prove upper bounds for the volume spectrum of a Riemannian manifold that depend only on the volume, dimension and a conformal invariant.
\end{abstract}

\maketitle

\section{Introduction}
Let $(M^n,g)$ be a closed Riemannian manifold of dimension $n\geq 2$. In \cite{Alm62}, Almgren proved that the space of mod 2 relative cycles $\mc Z_{n-1} (M;\mb Z_2)$ is weakly homotopic to $\mb{RP}^\infty$; see also \cite{LMN16}*{\S 2.5}. By performing a min-max procedure, Gromov \cite{Gro88} defined the {\em volume spectrum of $M$}, which is a sequence of non-decreasing positive numbers %$\{\omega_k(M;g)\}$ 
\[ 0<\omega_1(M, g)\leq\omega_2(M, g)\leq\cdots \leq\omega_k(M, g) \to \infty, \]
depending only on $M$ and $g$. Moreover, Gromov \cite{Gro88} also showed that for each $g$, $\omega_k$ grows like $k^{\frac{1}{n}}$; %has a sublinear growth of $k$; 
see also Guth \cite{Guth09}. 

For closed Riemannian surfaces (i.e. $n=2$), Y. Liokumovich \cite{Lio16} bounded all of the volume spectrum using the genus of the surfaces. In this paper, we generalize these results and prove conformal upper bounds for all of the volume spectrum of closed Riemannian manifolds.

\begin{theorem}\label{thm:main}
There exists a constant $C=C(n)$ such that for any $n$-dimensional closed Riemannian manifold $(M,g)$, we have 
\[    \omega_k(M,g)\leq C|M|_g^{\frac{n-1}{n}}\max\{k^{\frac{1}{n}},\mathrm{MCV}(M,g)^\frac{1}{n}\}.\]
\end{theorem}
Here $|\Sigma|_{g'}$ is denoted as the $\mc H^m$-measure with respect to $g'$ for any $m$-dimensional submanifold $\Sigma$ of $M$ and 
\[ \mathrm{MCV}(M,g):=\inf\{|M|_{g_0}: g_0 \text{ is a metric conformal to $g$ and $\Ric_{g_0}(M)\geq -(n-1)$}\},  \]
which is called the {\em min-conformal volume} of $M$; c.f. \citelist{\cite{Has11}*{\S 1}\cite{GL17}*{Definition 1.2}}. For simplicity, we use $[g]$ to denote the collection of Riemannian metrics that are conformal to $g$.

\begin{remark}
	We make several remarks here:
	\begin{enumerate}
		\item Note that by the uniformization theorem, $\mathrm{MCV}(M)\leq 2\gamma$ if $M$ is a closed surface of genus $\gamma$. Then Theorem \ref{thm:main} is exactly the same with \cite{Lio16}.
		\item From the proof, the estimates also hold for compact domains $N\subset M$ in Theorem \ref{thm:main}, i.e. for any $g_0\in[g]$ with $\Ric_{g_0}(M)\geq -(n-1)$ and $N\subset M$,
		\[ \omega_k(N,g)\leq C|N|_g^{\frac{n-1}{n}}\max\{k^{\frac{1}{n}},|N|_{g_0}^\frac{1}{n}\} .\]
		\item Our theorem is sharp in some sense. In general cases (not in a conformal class), L. Guth \cite{Guth09}*{Section 5} gave a counterexample to this question for the first width in the volume spectrum. In other words, a closed oriented Riemannian $n$-manifold may have volume 1 and arbitrarily large $\omega_1(M,g)$.
	\end{enumerate}
\end{remark}

\medskip
We point out that Glynn-Adey-Liokumovich \cite{GL17} proved conformal upper bounds for the first width in the volume spectrum (i.e. the case of $k=1$), which will be used in this paper. For closed Riemannian manifolds with non-negative Ricci curvature, the uniform upper bounds for the volume spectrum was proved by Glynn-Adey-Liokumovich \cite{GL17} and Sabourau \cite{Sab17}. Observe that $\mathrm{MCV}(M,g)=0$ provided that there exists $g_0\in[g]$ with $\Ric_{g_0}(M)\geq 0$. Hence we have the following corollary.
\begin{corollary}\label{cor:conformal to Ric geq 0}
Let $(M,g)$ be a closed Riemannian manifold and there exists $g_0\in[g]$ with $\Ric_{g_0}(M)\geq 0$. There exists a constant $C=C(n)$ such that
\[    \omega_k(M,g)\leq C|M|_g^{\frac{n-1}{n}}k^{\frac{1}{n}}.\]
\end{corollary}

\medskip
To understand the volume spectrum, Gromov \cite{Gro03}*{Remark 8.4} had an insightful idea that many properties of the eigenvalues of the Laplacian operators have analogs for the volume
spectrum. Furthermore, Gromov conjectured  that the volume spectrum $\{\omega_k (M,g)\}_{k\in\mb N}$ satisfy a Weyl's law, which has been fully proved by Liokumovich-Marques-Neves \cite{LMN16}. For Laplacian operators, Korevaar \cite{Kor93} proved the upper bounds for the Neumann eigenvalues of Riemannian manifolds which are conformal to a manifold with non-negative Ricci curvature. Later, Hassannezhad \cite{Has11} obtained the conformal upper bounds for the eigenvalues of the Laplacian in the conformal class of compact Riemannian manifolds. Our Theorem \ref{thm:main} and Corollary \ref{cor:conformal to Ric geq 0} are volume spectrum analogs of the results of Hassannezhad \cite{Has11} and Korevaar \cite{Kor93}, respectively.

We refer to \citelist{\cite{LY82}\cite{Bus82}\cite{CoMa}} for the estimates of the Laplacian operators and \citelist{\cite{NR04}\cite{BS10}\cite{LZ18}\cite{LM20}} for some developments of sweepouts by cycles.

\medskip
Due to the development of min-max theory by Almgren \citelist{\cite{Alm62}\cite{Alm65}}, Pitts \cite{Pi}, Schoen-Simon \cite{SS} and Marques-Neves \cite{MN16}, the volume spectrum bounds give information about finding minimal hypersurfaces in closed Riemannian manifolds; see \citelist{\cite{MN17}\cite{IMN17}\cite{Song18}\cite{MNS17}\cite{SZ20}}. In particular, using the Multiplicity One Theorem proven by X. Zhou \cite{Zhou20} (see also \cite{CM20}), Marques-Neves \cite{MN18} proved that in any closed Riemannian manifold $M$ of dimension $3\leq n\leq 7$, for generic metrics $g$, there exists a sequence of embedded minimal hypersurfaces $\{\Sigma_k\}$ such that 
\[   \omega_k(M,g)=\mc H^{n-1}(\Sigma_k)  \ \  \ \text{ and } \ \ \ \Index(\Sigma_k)=k.\]
Then our Theorem \ref{thm:main} gives a conformal upper bound for these embedded minimal hypersurfaces.

\smallskip
\subsection*{Idea of the proof} 
Let $(M,g)$ be a closed Riemannian manifold and $g_0\in[g]$ such that $\Ric_{g_0}(M)\geq -(n-1)$. Denote by $B_r^0(p)$ the geodesic ball in $M$ of radius $r$ and center $p$ with respect to $g_0$. For simplicity, we use $|\cdot |$ to denote $|\cdot|_{g}$.

We first recall the construction of $k$-sweepouts by Gromov \cite{Gro88} and Guth \cite{Guth09}*{Section 5}, where they proved that if a closed manifold $M$ is divided into a collection of open domains $\{V_j\}$, then
\[ \omega_k(M,g)\leq \Big|\bigcup_j\partial V_j\Big|+k\max_{j}\omega_1(V_j,g).\] 
Then the challenge is to divide $M$ into suitable domains for each $k$. Without loss of generality, we assume that $|M|=|M|_{g_0}$. By the work of Glynn-Adey-Liokumovich \cite{GL17}, it suffices to consider $k\geq |M|_{g_0}$ and $k>100^n$. We now fix $k$ and let $\alpha=|M|/k$ and $r=\alpha^{\frac{1}{n}}/C$. The aim is to bound $\omega_k(M,g)$ by $C|M|^{\frac{n-1}{n}}k^{\frac{1}{n}}$.

In the first step, we subdivide $M$ into domains $\{D_j\}_{j=1}^{m+1}$ ($m\leq k-1$) such that $|D_j|_{g_0}<1$ for $1\leq j\leq m$, $\sum|\partial D_j|\sim k^{\frac{1}{n}}$ and $|B_r^0(p)\cap D_{m+1}|\leq \alpha$ for all $p\in M$. This can be done inductively by taking $B^0_r(p)$ such that its $g$-volume is larger than $\alpha$. Then the length-area method also enables us to control $|\partial D_j|$; see Claim \ref{claim:thin-thick decomposition} for details.

The next step is to subdivide $D_{m+1}$. To do this, we always take $B_r^0(p)$ that has the largest area in the remaining part with respect to $g$. Then the length-area method allows us to find a domain $V_i$ between $B_{3r}^0(p)$ and $B_{4r}^0(p)$ such that its boundary has a desired bound. The difficulty here is that $|B_{4r}^0(p)|_{g_0}$ is used to bound $|\partial V_i|$. And these balls of radius $4r$ will intersect each other. To overcome this, we proved that for each point $x\in D_{m+1}$, the number of $V_i$ that contains $x$ is bounded by a uniform constant depending only on $n$. Then using the H\"older's inequality, we obtain the desired covers for $D_{m+1}$.

Finally, we are going to subdivide $D_j$ for $1\leq j\leq m$. One of the key ingredients is the isoperimetric inequality developed by Glynn-Adey-Liokumovich \cite{GL17}*{Theorem 3.4} (see also Theorem \ref{thm:isoperi by GL}), which allows us to subdivide $D_j$ into two parts. Repeating this process, we finally subdivide $D_j$ into $\{U_i^j\}_i$ until each small domain has $g$-volume bounded by $|M|/k$. Then using the estimates for the first width in the volume spectrum in \cite{GL17} (see also Theorem \ref{thm:good cover for 1-width} for compact domains), $k\omega_1(U_i^j,g)$ is naturally bounded by $k^{\frac{1}{n}}$. It remains to bound the boundary of $U_i^j$ that lies in $\Int D_j$, which are exactly the isoperimetric hypersurfaces in Theorem \ref{thm:isoperi by GL}. In Subsection \ref{subsec:tree}, a general way will be developed to study this kind of {\em tree decomposition}; see Proposition \ref{prop:linear growth of N} for details. We would like to emphasize that $|D_j|<1$ is crucial to have the desired bounds in this part.

\subsection*{Outline}
This paper is organized as follows. Section \ref{sec:pre} includes some results that will be used in this paper and an upper bound for the tree decomposition. In Section \ref{sec:thin} and \ref{sec:thick}, we provide the details to subdivide the conformally thin and thick domains, respectively. Finally, Section \ref{sec:final bounds} is devoted to prove the main theorem. We also give more details of the proof of Theorem \ref{thm:good cover for 1-width} in Appendix \ref{sec:proof of GL-theorem}.

\subsection*{Acknowledgments} We are grateful to Professor Yevgeny Liokumovich for bringing this problem to our attention and many valuable discussions.

\begin{comment}

Idea of the proof:
\begin{enumerate}
\item  Let $|M|_g=|M|_{g_0}$, $r=r_k\sim (|M|_g/k)^{\frac{1}{n}}$, $\alpha=\alpha_k\sim |M|_g/k$.
\item If $|B_r(p)|_{g_0}>\alpha$, then we can covered $B_r(p)\subset D\subset B_{4r}(p)$ by $\{U_j\}$ such that 
\[  |\bigcup \partial U_j|_g+ k_j\omega_1(U_j,g)\leq C|D|_g^{\frac{n-1}{n}}k_j^{\frac{1}{n}}.\]
where $k_j$ is almost $k|D|_g/|M|_g$
The key point for this part is that $|D|_{g_0}\leq 1$. 

Then  note that we have only at most $k$ many this kind of $D$. Moreover, we can control the $|\partial D|_g$.
\item After removing this kind if $D$, the rest part is denoted by $D_{m+1}$, which satisfying 
\[ |B_r(p)\cap D_{m+1}|_g \leq \alpha.\]
Then we cover $D_{m+1}$ by finitely many $V_j$ with $B_r(x)\subset V_j\subset B_{4r}(x)$ and 
	\begin{gather} 
\Big|\bigcup_{j=1}^{L}\partial V_j\cap \mathrm{Int}M\Big|_g\leq (C_1/r)\cdot |D_{m+1}|^{\frac{1}{n}}_{g_0}\cdot |D_{m+1}|_g^{\frac{n-1}{n}};\\
 \omega_1(V_j)\leq C_1\alpha^{\frac{n-1}{n}}\ \ \text{ for } \ \ 1\leq j\leq L,
\end{gather}
Then we also have the bound in this domain.
	\end{enumerate}
\end{comment}

\section{Preliminary}\label{sec:pre}

\subsection{Notations}
In this paper, $(M^n,g)$ is always a closed Riemannian manifold with dimension $n$ and $N$ is a compact domain in $M$ with piecewise smooth boundary. 

We now recall the formulation in \cite{LMN16}. Let $(N,\partial N,g)\subset \mb R^L$ be a compact Riemannian manifold with piecewise smooth boundary. Let $\mc R_k(N;\mb Z_2)$ (resp. $\mc R_{k}(\partial N;\mb Z_2)$) be the space of $k$-dimensional rectifiable currents in $\mb R^L$ with coefficients in $\mb Z_2$ which are supported in $N$ (resp. $\partial N$). Denote by $\mf M$ the mass norm. Let 
\begin{equation}
\label{eq:def from integer rect}
Z_k(N,\partial N;\mb Z_2):=\{T\in \mc{R}_k(N;\mb Z_2) : \spt(\partial T)\subset \partial N \}.
\end{equation}
We say that two elements $S_1,S_2\in Z_k(N,\partial N;\mb Z_2)$ are equivalent if $S_1-S_2\in \mc{R}_k(\partial N;\mb Z_2)$. Denote by $\mc{Z}_k(M,\partial N; \mb Z_2)$ the space of all such equivalence classes. The mass and flat norms for any $\tau\in \mc{Z}_k(N,\partial N;\mb Z_2)$ are defined by 
\[\mf{M}(\tau):=\inf\{\mf{M}(S) : S\in \tau\}\quad \text{ and }\quad \mc{F}(\tau):=\inf\{\mc{F}(S) : S\in \tau\}.\]
The support of $\tau\in \mc{Z}_k(M,\partial M;\mb Z_2)$ is defined by 
\[\spt (\tau):=\bigcap_{S\in \tau}\spt (S).\]

Let $X$ be a finite dimensional simplicial complex. Given $p\in \mb N$, a continuous map in
the flat topology
\[\Phi : X \rightarrow \mc Z_{n} (N, \partial N;\mb Z_2)\]
is called a {\em $k$-sweepout} if the $k$-th cup power of $\lambda = \Phi^*(\bar{\lambda})$ is non-zero in $H^{k}(X; \mb Z_2 )$ where $0 \neq \bar{\lambda} \in H^1(\mc Z_{n} (N, \partial N;\mb Z_2);\mb Z_2) \cong \mb Z_2$. Denote by $\mc P_k(N)$ the set of all $k$-sweepouts that are continuous in the flat topology and {\em have no concentration of mass} (\cite{MN17}*{\S 3.7}), i.e.
\[\lim_{r\rightarrow 0}\sup\{\mf M(\Phi(x)\llcorner B_r(q)):x\in X,q\in M\}=0.\]

In \cite{MN17} and \cite{LMN16}, the {\em $k$-width of codimension one} is defined as
\begin{equation}
\omega_k(N,g):=\inf_{\Phi\in\mc P_k}\sup\{\mf M(\Phi(x)):x\in \mathrm{dmn}(\Phi)\}.
\end{equation}
$\{\omega_k(N,g)\}$ are also called the {\em volume spectrum}.

\begin{remark}
In this paper, we used the integer rectifiable currents, which is the same with \cite{LZ16}. However, the formulations are equivalent to that in \cite{LMN16}; see \cite{GLWZ19}*{Proposition 3.2} for details.
\end{remark}

\subsection{Conformal bounds for the first width}
In \cite{GL17}, Glynn-Adey and Liokumovich proved the uniform bound of the first width for all closed manifolds. With minor modification, their arguments can be applied for compact domains.  Such a uniform bound will be used in this paper later.

Let $g$ and $g_0$ be two Riemannian metrics on $M$. For any $m$-dimensional submanifold $\Sigma$ of $M$, we use $|\Sigma|$ and $|\Sigma|_{g_0}$ to denote the $\mc H^m$-measure with respect to $g$ and $g_0$.
\begin{theorem}[Glynn-Adey-Liokumovich \cite{GL17}]\label{thm:good cover for 1-width}
	Let $N$ be a compact domain of a closed Riemannian manifold $(M,g)$ with dimension $n$. Let $g_0$ be another metric on $M$ which is conformal to $g$ and $\Ric_{g_0}(M)\geq -1$. There exists a constant $K$ depending only on the dimension of $N$ such that 
	\[  \omega_1(N,g)\leq K\cdot |N|^{\frac{n-1}{n}}(1+|N|^{\frac{1}{n}}_{g_0}).\]
\end{theorem}
For completeness, we sketch the idea of the proof here and give more details in Appendix \ref{sec:proof of GL-theorem}.

We first handle the case that $N$ has smooth boundary. Following the steps in \cite{GL17}, we decompose the domain with small volume into small pieces so that the argument in \cite{GL17}*{Proposition 2.3} can be applied, and then we use the inductive method in \cite{GL17}*{Theorem 5.1}. 

To decompose $D\subset N$ with small volume, we will cut the part intersecting $\partial N$. Then the regularity theory \cite{Mor03}*{Theorem 4.7} (see also \cite{GLWZ19}*{Theorem 4.7}) for the free boundary minimizing problem is used. In order to show that such a minimizing hypersurface does not intersect a smaller ball, we employ the monotonicity formula in \cite{GLZ16}*{Theorem 3.4}.

Finally, for compact domain with piecewise smooth boundary, we can take a tubular neighborhood $U$ with $|U|\leq 2|N|$ and $|U|_{g_0}\leq 2|N|_{g_0}$ and $U$ has smooth boundary. Then the desired inequality follows from $\omega_1(N,g)\leq \omega_1(U,g)$.

\subsection{The length-area method}
Let $(M,g)$ be a closed Riemannian manifold and $N$ be a compact domain with piecewise smooth boundary. Let $g_0$ be a metric on $M$ which is conformal to $g$ and $\Ric_{g_0}(M)\geq -(n-1)$. Denote by $\nabla$ and $\nabla^0$ the Levi-Civita connection with respect to $g_0$ and $g$. For any compact domain $D\subset M$, denote by 
\[\mc N_r^0(A):=\{x\in M:\dist_{g_0} (x,A)\leq r\},\]
where $\dist_{g_0}(\cdot,\cdot)$ is the distance with respect to $g_0$. Recall that $|\Sigma|$ and $|\Sigma|_{g_0}$ are used to denote the $\mc H^m$-measure with respect to $g$ and $g_0$ if $\Sigma$ is a $m$-dimensional submanifold of $M$.

The following inequality is from the well-known {\em length-area method} (see \citelist{\cite{Gro83}*{\S 5}\cite{GL17}*{Theorem 3.4}\cite{Lio16}*{Lemma 4.1}}) and will be used in this paper. 
\begin{proposition}\label{prop:length area}
For any compact domain $D\subset N$ and $r>0$, there exists a compact domain $V$ of $N$ satisfying $D\subset V\subset \mc N_r^0(D)$ and 
\[ |\partial U\cap \Int N|\leq (1/r)\cdot |N\cap \mc N^0_r(D)\setminus D|_{g_0}^{\frac{1}{n}}\cdot |N\cap \mc N^0_r(D)\setminus D|^{\frac{n-1}{n}}.\] 
\end{proposition}
\begin{proof}
We present the proof in \cite{GL17}*{Theorem 3.4} here. For $x\in M$, denote by 
\[f(x)=\dist_{g_0}(x,D).  \]
By the co-area formula,
\begin{align*}
\int_0^r|f^{-1}(t)\cap\Int N|dt&=\int_{f^{-1}(0,r)\cap N}|\nabla f| d\mc H^n(g)\\
&\leq \Big(\int_{f^{-1}(0,r)\cap N}|\nabla f|^nd \mc H^n(g)\Big)^{\frac{1}{n}}\cdot |f^{-1}(0,r)\cap N|^{\frac{n-1}{n}}\\
&=|f^{-1}(0,r)\cap N|^{\frac{1}{n}}_{g_0}\cdot |f^{-1}(0,r)\cap N|^{\frac{n-1}{n}}.
\end{align*}
Here the last equality follows from $|\nabla f|^nd\mc H^n(g)=|\nabla^0 f|^nd\mc H^n(g_0)$. Note that 
\[f^{-1}(0,r)= \mc N^0_r(D)\setminus D.\]
Hence Proposition \ref{prop:length area} is proved.
\end{proof}

\subsection{Tree decomposition}\label{subsec:tree}
Let $\alpha=\overline{\alpha_1\alpha_2\cdots \alpha_m}$ be an ordered binary array with $\alpha_j\in\{0,1\}$. Then we define $|\alpha|=m$. For two binary arrays $\alpha$ and $\beta$, we say $\alpha \preceq \beta$ if $\alpha_j=\beta_j$ for all $j\leq |\alpha|$. We say $\Lambda$ is an {\em admissible tree} provided the following holds:
\begin{itemize}
	\item if $\alpha\in\Lambda$, then $\beta\in\Lambda$ for any $\beta\preceq \alpha$;
	\item $\overline{\alpha0}\in \Lambda$ if and only if $\overline{\alpha1}\in\Lambda$;
\end{itemize}  
Denote by $\partial \Lambda= \{\alpha\in\Lambda: \text{ if } \beta\in\Lambda \text{ with } \alpha\preceq\beta, \text{ then } \beta =\alpha\}$.

Let $\Lambda$ be an admissible tree and $\lambda\in(0,1/2]$. For any positive real number $X\geq 1$, we say a sequence of real numbers $\{X_\alpha\}$ is a {\em$(\Lambda, \lambda)$-decomposition} if 
\begin{itemize}
	\item $X=X_0+X_1$ and $X_i> \lambda X$ for $i\in\{0,1\}$;
	\item $X_\alpha=X_{\oli{\alpha0}}+X_{\oli{\alpha1}}$ and $X_{\oli{\alpha i}}\geq \lambda X_\alpha$ for all $\alpha\in\Lambda\setminus \partial\Lambda$ and $i\in\{0,1\}$;
	\item $X_\alpha\geq 1$ for all $\alpha\in \Lambda$.
\end{itemize}

In this section, $\lambda\in (0,1/2)$ is a constant. Let
\[  \wti \lambda =\big[\lambda^{\frac{n-1}{n}}+(1-\lambda)^{\frac{n-1}{n}}-1\big]^{-1}.\]
Then for any $t\in[\lambda,1-\lambda]$, we have
\begin{equation}\label{eq:for wti lambda}
\wti \lambda \cdot \big(t^{\frac{n-1}{n}}+(1-t)^{\frac{n-1}{n}}-1\big)\geq 1.
\end{equation}
For any $X\geq 1$, we define
\[  \mc N(X):= \sup\Big\{ X^{\frac{n-1}{n}}+\sum_{\alpha\in\Lambda} X_\alpha^{\frac{n-1}{n}}: \{X_\alpha\} \text{ is a $(\Lambda,\lambda)$-decomposition for some admissible tree } \Lambda \Big \}.\]

The main result in this subsection is that $\mc N(X)$ has linear growth.
\begin{proposition}\label{prop:linear growth of N}
	For any $X\geq1$, we have
	\[ \mc N(X)+\wti \lambda X^{\frac{n-1}{n}}\leq (1+\wti \lambda) X. \]
\end{proposition}
\begin{proof}
	For any $(\Lambda,\lambda)$-decomposition $\{X_\alpha\}$, we have that
	\[  X^{\frac{n-1}{n}}+\sum_{\alpha\in\Lambda} X_\alpha^{\frac{n-1}{n}}\leq X^{\frac{n-1}{n}}+\mc N(X_0)+\mc N(X_1)\leq X^{\frac{n-1}{n}}+\sup_{\lambda\leq t\leq 1-\lambda} (\mc N(tX)+\mc N((1-t)X)).  \]
	This implies that 
	\[  \mc N(X)\leq X^{\frac{n-1}{n}}+\sup_{\lambda\leq t\leq 1-\lambda} (\mc N(tX)+\mc N((1-t)X)).\]
	Denote by $\wti {\mc N}(X)=\mc N(X)+\wti \lambda X^{\frac{n-1}{n}}$. Then 
	\begin{equation}\label{eq:N convex}
	 \wti{\mc N}(X)\leq \sup_{\lambda\leq t\leq 1-\lambda} (\wti {\mc N}(tX)+\wti {\mc N}((1-t)X)), 
	 \end{equation}
	where we used the fact \eqref{eq:for wti lambda}.
	For any $X\in[1,2)$, we have $\wti {\mc N}(X)=X^{\frac{n-1}{n}}+\wti \lambda X^{\frac{n-1}{n}}\leq (1+\wti \lambda )X$. Now we prove the inequality inductively. Suppose that it holds true for $X< Y$ ($Y\geq 2$). Then for any $X\in[Y,Y+\lambda]$ and $t\in[\lambda,1-\lambda]$, we have
	\[   tX\leq (1-\lambda)(Y+\lambda )\leq Y-\lambda.\]
	Hence 
	\[\wti{\mc N}(tX)\leq (1+\wti \lambda )tX \ \ \ \text { and } \ \ \ \wti{\mc N}((1-t)X)\leq (1+\wti \lambda )(1-t)X,\]
Together with \eqref{eq:N convex}, we conclude that
	\[  \wti {\mc N}(X)\leq  \sup_{\lambda\leq t\leq 1-\lambda} \big[(1+\wti \lambda)tX+(1+\wti \lambda)(1-t)X\big]=(1+\wti \lambda)X. \]
	This finishes the proof of Proposition \ref{prop:linear growth of N}.
\end{proof}

\section{Dividing conformally thin domains}\label{sec:thin}
Let $(M,g)$ be a closed Riemannian manifold and $g_0\in[g]$ with $\Ric_{g_0}(M)\geq-(n-1)$.  Denote by $B^0_r(p)$ the geodesic ball in $(M,g_0)$ with center $p$ and radius $r$. In this section, we divide the compact domains $N$ that geodesic balls in $(M,g_0)$ of radius $r$ satisfying
\[ |B_r^0(p)\cap N|\leq \alpha,  \ \ \  \forall p\in M,\]
 where $r$ and $\alpha$ are given constants. This kind of domains are called to be {\em conformally thin}. 

Denote by $v(r,n)$ the volume of the geodesic ball in an $n$-dimensional hyperbolic manifold (with sectional curvature $-1$). Denote by
\[ C(r)=\max_{0<t\leq r}\Big \{ 1+\Big[\frac{v(9r/2,n)}{v(r/2,n)}\Big] \Big\} . \]
Then in any complete Riemannian manifold with $\Ric\geq -(n-1)$, every geodesic ball of radius $4s$ can be covered by $C(r)$ many balls of radius $s$ for all $s\in(0,r]$. Note that $C(r)$ is a constant depending only on $r$ and $n$; c.f. \cite{CoMa}*{Example 2.1}.

Let $C_0=C_0(n)$ be the constant such that for $r<10$,
\[  v(r,n)\leq C_0r^n.\]
By the classical Bishop-Gromov inequality, a geodesic ball with radius $r<10$ in a Riemannian manifold with $\Ric\geq-(n-1)$ has the volume bounded by $C_0r^n$ from above. Let $K$ be the constant in Theorem \ref{thm:good cover for 1-width}.

\begin{lemma}\label{lem:good decomposition of thin part}
Let $N$ be a compact domain with (possibly empty) piecewise smooth boundary in some closed Riemannian manifold $(M,g)$. Suppose that there exist $\alpha>0$ and $r\in(0,1)$ satisfying 
	\[|B^0_r(p)\cap N|\leq \alpha\]
	for all $p\in M$. Then $N$ can be divided into finitely many open domains $\{ V_j \}_{j=1}^L$ by $\cup \partial V_j$ satisfying
	\begin{gather} 
	\label{item:bound boundary} \Big|\bigcup_{j=1}^{L}\partial V_j\cap \mathrm{Int}N\Big|\leq (C_1/r)\cdot |N|^{\frac{1}{n}}_{g_0}\cdot |N|^{\frac{n-1}{n}};\\
\label{item:bound volume} \omega_1(V_j,g)\leq C_1\alpha^{\frac{n-1}{n}}\ \ \text{ for } \ \ 1\leq j\leq L,
	\end{gather}
where $C_1=C(r/2)C(r)+(4C_0+1)K\cdot C(r)$.
\end{lemma}
\begin{proof}
Since $\Ric_{g_0}(M)\geq -(n-1)$ and $r<1$, then we have 
\begin{equation}\label{eq:bound B4r for g0}	
|B_{4r}^0(p)|_{g_0}\leq C_0(4r)^n.
	\end{equation}

Now we construct $\{V_j\}$ inductively. Let $V_0=\emptyset$. Suppose we have $V_0,\cdots,V_j$ and $M\setminus \cup_{i=1}^j\overline V_i\neq \emptyset$. Then we take $p_{j+1}\in M\setminus \cup_{i=0}^j\overline V_i$ such that for all $p\in M$,
\[ \Big|N\cap B_r^0(p_{j+1})\setminus \bigcup_{i=0}^jV_i\Big|\geq\Big |N\cap B_r^0(p)\setminus \bigcup_{i=0}^jV_i\Big| .\]
Note that $B_{4r}^0(p_{j+1})$ is covered by $C(r)$ many balls of radius $r$. It follows that 
\[ \Big|N\cap B_{4r}^0(p_{j+1})\setminus \bigcup_{i=0}^jV_i\Big|\leq C(r)\Big|N\cap B_r^0(p_{j+1})\setminus \bigcup_{i=0}^jV_i\Big|\leq C(r)\alpha.\]
Then by Proposition \ref{prop:length area}, we take $V_{j+1}$ satisfying
\begin{equation}\label{eq:construct Vj+1}
B_{3r}^0(p_{j+1})\cap N\setminus \bigcup_{i=0}^jV_i\subset V_{j+1}\subset B_{4r}^0(p_{j+1})\cap N\setminus \bigcup_{i=0}^jV_i
\end{equation}
and
\begin{equation} \label{eq:bound boundary of Vj}
\Big|\partial V_{j+1}\cap\Int (N\setminus \bigcup_{i=0}^jV_i)\Big|\leq \frac{1}{r}\cdot |B_{4r}^0(p_{j+1})\cap N|^{\frac{n-1}{n}}\cdot |B_{4r}^0(p_{j+1})\cap N|_{g_0}^{\frac{1}{n}} .
\end{equation}
By Theorem \ref{thm:good cover for 1-width},
\begin{equation}\label{eq:bound 1-width for Uj}
 \omega_1(V_{j+1},g)\leq K|V_{j+1}|^{\frac{n-1}{n}}(1+\big|B_{4r}^0(p_{j+1})\big|_{g_0}^{\frac{1}{n}})\leq (4C_0+1)K\cdot C(r)\alpha^{\frac{n-1}{n}}.
 \end{equation}
 Here we used \eqref{eq:construct Vj+1} and \eqref{eq:bound B4r for g0} in the last inequality. 
 
Observe that $p_{j+1}\notin B_{2r}(p_i)$ for $i\leq j$, which implies that 
 \[ B_r^0(p_{j+1})\cap B_r^0(p_i) =\emptyset, \ \ \ 1\leq i\leq j.\]
Then there exists $L\geq 1$ such that 
 \[   N=\bigcup_{j=1}^L\overline V_j.\]
 It remains to prove that these open sets satisfy our requirements. We first prove that every $x\in M$ is contained in at most $C(r/2)\cdot C(2r)$ many $V_j$. Namely, if $x\in V_j$, then   $B^0_r(p_j)\subset B_{5r}^0(x)$. Now let
 \[  J(x)= \#\{V_j:1\leq j\leq L \text{ and } B_r^0(p_j)\subset B_{5r}^0(x) \}.\]
 Note that $B_{5r}^0(x)$ can be covered by $C(r/2)C(2r)$ many balls $\{B^0_{r/2}(z_i)\}$ in $M$. By taking $z_j$ such that $p_j\in B^0_{r/2}(z_j)$, then we have $B^0_{r/2}(z_j)\subset B_{r}^0(p_j)$. Thus $J(x)\leq C(r/2)C(2r)$.

 By \eqref{eq:bound boundary of Vj}, we have 
 \begin{align*} 
 \Big|\bigcup_{j=1}^{L}\partial V_j\cap \mathrm{Int}N\Big|&=\sum_{j=0}^{L-1}
\Big|\partial V_{j+1}\cap\Int (N\setminus \bigcup_{i=0}^jV_i)\Big|\\ 
&\leq   \sum_{j=1}^L \frac{1}{r}\cdot |B_{4r}^0(p_{j})\cap N|^{\frac{n-1}{n}}\cdot |B_{4r}^0(p_{j})\cap N|_{g_0}^{\frac{1}{n}} \\
 &\leq \frac{1}{r} \cdot \Big(\sum_{j=1}^L|B_{4r}^0(p_{j})\cap N|\Big)^{\frac{n-1}{n}}\cdot\Big(\sum_{j=1}^m |B_{4r}^0(p_{j})\cap N|_{g_0}\Big)^{\frac{1}{n}}\\
 &\leq C(r/2)C(r)\cdot \frac{1}{r}\cdot |N|^{\frac{n-1}{n}}\cdot |N|_{g_0}^{\frac{1}{n}}.
 \end{align*}
Together with \eqref{eq:bound 1-width for Uj}, Lemma \ref{lem:good decomposition of thin part} follows by taking $C_1=C(r/2)C(r)+(4C_0+1)K\cdot C(r)$.

\end{proof}

\section{Dividing conformally thick domains}\label{sec:thick}
Let $(M,g)$ be a closed manifold and $g_0\in [g]$ such that $\Ric_{g_0}(M)\geq -(n-1)$. Let $N$ be a compact domain in $M$ with piecewise smooth boundary. In this section, we estimates the volume spectrum of the domains with small $g_0$-volume.

We first recall the isoperimetric inequality developed by Glynn-Adey-Liokumovich in \cite{GL17}, which is a consequence of the length-area method.
\begin{theorem}[\cite{GL17}*{Theorem 3.4}]\label{thm:isoperi by GL}
There exists a constant $c(n)$ such that the following holds: Let $U \subset M$ be an open subset. There exists an $(n-1)$-submanifold $\Sigma\subset U$ subdividing $U$ into two open sets $U_1$ and $U_2$ such that $|U_i| \geq 25^{-n}|U|$ and $|\Sigma| \leq c(n)\max\{1,|U|_{g_0}^{\frac{1}{n}}\}|U|^{\frac{n-1}{n}}$.
\end{theorem}

Now we are ready to prove the main result of this section.

\begin{theorem}\label{thm:cover of thick part}
	There exists $C_2=C_2(n)$ satisfying the following: for every positive integer $k$, each closed $n$-dimensional Riemannian manifold $(M,g)$ and compact domain $N\subset M$ with $|N|_{g_0}\leq 1$, there exists a collection of compact domains $\{U_j\}$ such that 
	$ N=\cup \oli{U}_j$	and 
\begin{equation}
 \big|\cup_j\partial U_j\cap(\mathrm{Int}N)\big|+ k\max_{j}\omega_1(U_j,g)\leq C_2|N|^{\frac{n-1}{n}}k^{\frac{1}{n}}.
\end{equation}
 As a corollary, $\omega_k(N,g)\leq C_2|N|^{\frac{n-1}{n}}k^{\frac{1}{n}}$.
\end{theorem}

\begin{proof}
	Without loss of generality, we assume that $|N|=1$. Let $K>1$ be the constant in Theorem \ref{thm:good cover for 1-width}. Then every compact domain $N'\subset M$ satisfies
	\begin{equation}\label{eq:bound for small k}
 \omega_1(N',g)\leq K|N'|^{\frac{n-1}{n}}.
	\end{equation}
	
	Now let $k>50^n$. Then by Theorem \ref{thm:isoperi by GL}, there exists an $(n-1)$-submanifold $\Sigma\subset M$ subdividing $N$ into two open sets $M_0$ and $N_1$ such that $ |N_1|\geq |N_0|\geq 1/25^n$ and $|\Sigma|\leq  c(n)\max\{1,|N|^{\frac{1}{n}}_{g_0}\}=c(n)$.
	
	Let $\overline{N}_\alpha$ be a compact domain of $N$, where $\alpha=\overline{i_1i_2\cdots i_{|\alpha|}}$ and $i_j\in\{0,1\}$. If $|N_\alpha|\geq 50^n/k$, then using Theorem \ref{thm:isoperi by GL} again, there exists an $(n-1)$-submanifold $\Sigma_\alpha$ subdividing $N_\alpha$ into two open sets $N_{\overline{\alpha0}}$ and $N_{\oli{\alpha1}}$ such that $|N_{\oli{\alpha 1}}|_g\geq |N_{\oli{\alpha0}}|\geq |N_\alpha|/25^n$ and 
	\begin{equation}\label{eq:bound Sigma_alpha}
	|\Sigma_\alpha|\leq c(n)|N_\alpha|^{\frac{n-1}{n}}\max\{1,|N_\alpha|_{g_0}^{\frac{1}{n}}\}=c(n)|N_\alpha|^{\frac{n-1}{n}}.
	\end{equation} 
	Note that we always have $k|N_{\alpha 1}|\geq k|N_{\oli{\alpha0}}|\geq k|N_\alpha|/25^n\geq 2^n$.
	
	Denote by $\Lambda$ the collection of $\alpha$ appeared in the previous process. Then $\Lambda$ is an admissible tree (see Subsection \ref{subsec:tree}). Recall that 
	\[\partial \Lambda=\{\alpha\in \Lambda: \overline{\alpha0}\notin\Lambda\}.\]
	Then we have 
	\[ N=\bigcup_{\alpha\in \partial \Lambda}N_\alpha, \]
	where $|N_\alpha|< 50^n/k$ and $\mathrm{Int}N_\alpha\cap \mathrm{Int} N_\beta=\emptyset$ for any $\alpha\neq \beta \in\partial \Lambda$.

Now we define $\{U_j\}$ as $\{N_\alpha\}_{\alpha\in\partial \Lambda}$. Then we prove that such a collection of domains satisfy our requirements. Denote by $k_\alpha=k|N_\alpha|$. Note that $|N_\alpha|< 50^n/k$. Then for each $\alpha\in\partial \Lambda$,
	\[ k_\alpha= k|N_\alpha|<50^n. \]
By \eqref{eq:bound for small k}, we have 
\[   k_\alpha\omega_1(N_\alpha,g)<  50^n\omega_1(N_\alpha,g)\leq  50^nK|N_\alpha|^{\frac{n-1}{n}}, \]
which implies that for all $\alpha\in\partial \Lambda$,
\begin{equation} \label{eq:for general omegak}
  k\omega_1(N_\alpha,g)\leq k/k_\alpha \cdot 50^nK|N_\alpha|^{\frac{n-1}{n}} = k^{\frac{1}{n}}/k_\alpha\cdot 50^nK\cdot (k_\alpha)^{\frac{n-1}{n}}\leq 50^nK\cdot k^{\frac{1}{n}}
\end{equation}
Here in the equality, we used $ k|N_\alpha|= k_\alpha$.

	\begin{claim}\label{claim:bound Sigma}
		There exists $K_1(n)$ depending only on $n$ such that 
		\[  \Big|\bigcup_{\alpha\in\partial \Lambda}\partial N_\alpha\cap \Int N\Big|\leq K_1(n)k^{\frac{1}{n}}.\]
	\end{claim}
	\begin{proof}[Proof of Claim \ref{claim:bound Sigma}]
	Note that $\{k_\alpha\}$ is a $(1/50^n,\Lambda)$-decomposition. Then by Proposition	\ref{prop:linear growth of N} (by letting $\lambda=1/50^n$),
	\begin{equation}\label{eq:using decomposition}
	k^{\frac{n-1}{n}}+\sum_{\alpha\in\Lambda\setminus\partial \Lambda} k_\alpha^{\frac{n-1}{n}}\leq (1+\wti \lambda)k,
	\end{equation}
	where  $\wti \lambda$ is defined by \eqref{eq:for wti lambda}. Then we have
		\begin{align*}
		\Big|\bigcup_{\alpha\in\partial \Lambda}\partial N_\alpha\cap \Int N\Big|&=|\Sigma|+\sum_{\alpha\in\Lambda\setminus \partial \Lambda}|\Sigma_\alpha|\\
		&\leq  c(n)\Big(|N|^{\frac{n-1}{n}}+\sum_{\alpha\in\Lambda\setminus\partial \Lambda}|N_\alpha|^{\frac{n-1}{n}}\Big)\\
		&\leq c(n)k^{-\frac{n-1}{n}}\Big( k^{\frac{n-1}{n}}+\sum_{\alpha\in\Lambda\setminus\partial \Lambda} k_\alpha^{\frac{n-1}{n}} \Big)\\
		&\leq 2 c(n)\cdot k^{\frac{1}{n}}(1+\wti \lambda).
		\end{align*}
		Here the first inequality is from \eqref{eq:bound Sigma_alpha}; the last one follows from \eqref{eq:using decomposition}.
		
		Let 
		\[  K_1(n)=	2c(n)(1+\wti \lambda).\]
		Then Claim \ref{claim:bound Sigma} is proved.
	\end{proof}% end of Claim bound Sigma
	
Recall that $\{U_j\}$ are exactly $\{N_\alpha\}_{\alpha\in\partial \Lambda}$. Using Claim \ref{claim:bound Sigma} and \eqref{eq:for general omegak}, we obtain
	\[  	 |\cup_j\partial U_j\cap \Int N|+k\max_j\omega_1(U_j,g)\leq (50^nK+K_1(n)) k^{\frac{1}{n}} . \]
	This is the desired inequality by letting $C_2=50^nK+K_1(n)$.
\end{proof}

\section{The conformal upper bounds}\label{sec:final bounds}
In this section, we prove the conformal upper bounds for the volume spectrum. We will first divide the manifold into conformally thin and thick domains and then Lemma \ref{lem:good decomposition of thin part} and Theorem \ref{thm:cover of thick part} can be applied respectively.

Recall that $|\cdot|$ and $|\cdot|_{g_0}$ are denoted as the Hausdorff measure with respect to $g$ and $g_0$. The following result is equivalent to Theorem \ref{thm:main}.
\begin{theorem}
	There exists a constant $C_3=C_3(n)$ such that for any $n$-dimensional closed Riemannian manifold $(M,g)$, we have 
	\[    \omega_k(M,g)\leq C_3|M|^{\frac{n-1}{n}}\max\{k^{\frac{1}{n}},|M|_{g_0}^\frac{1}{n}\},\]
	where $g_0$ is conformal to $g$ and $\Ric_{g_0}(M)\geq -(n-1)$.
\end{theorem}
\begin{proof}
Without loss of generality, we assume that $|M|(:=|M|_g)=|M|_{g_0}$. For any $k>100^n$, define
\[  r_k=\frac{1}{4}\cdot \Big(\frac{|M|}{2kC_0C(1)}\Big)^{\frac{1}{n}}, \ \ \ \text{ and } \ \ \ \alpha_k=\frac{|M|}{k}. \]	
Denote by 
\[  \bar k=\Big[\frac{|M|}{2C(1)}\Big]+1. \]
Then for any $k\geq \bar k$, we have $r_k< 1/4$.	
\begin{claim}\label{claim:thin-thick decomposition}
There exist  $m(\leq k-1)$ many domains $\{D_j\}_{j=1}^m$ such that 
\begin{itemize}
	\item $ |D_j|_{g_0}< 1$ and $|\cup\partial D_j|\leq 4C_0 C(1)|M|^{\frac{n-1}{n}}\cdot k^{\frac{1}{n}}$;
	\item $|B_{r_k}^0(p)\setminus \cup_{j=1}^m D_j|<\alpha_k$ for all $p\in M$.
\end{itemize}
\end{claim}
\begin{proof}[Proof of Claim \ref{claim:thin-thick decomposition}]
Let $D_0=\emptyset$. Then we construct $\{D_j\}$ inductively. Suppose we have $D_0,\cdots, D_j$. If $\big|B_{r_k}^0(p)\setminus \cup_{i=1}^jD_i\big|<\alpha_k$ for all $p\in M $, then we just let $m=j$. Otherwise, take $p_{j+1}$ such that for all $p\in M$,
\[\Big|B_{r_k}^0(p_{j+1})\setminus \bigcup_{i=0}^jD_i\Big|\geq \Big|B_{r_k}^0(p)\setminus \bigcup_{i=0}^jD_i\Big|.\]
Clearly, $p_{j+1}\notin B_{2r_k}^0(p_i)$ for all $i\leq j$ and $|B_{r_k}^0(p_{j+1})\setminus \cup_{i=1}^jD_i|_g\geq\alpha_k$. Note that $B_{4r_k}^0(p_{j+1})$ is covered by $C(r_k)$ many balls of radius $r_k$. Thus we have
\begin{equation}\label{eq:bound 4r ball by r ball}
 \Big|B_{4r_k}^0(p_{j+1})\setminus \bigcup_{i=0}^jD_i\Big|\leq C(r_k)\Big|B_{r_k}^0(p_{j+1})\setminus \bigcup_{i=0}^jD_i\Big|.  
 \end{equation}
Since $r_k<1$, we have that  
\begin{equation}\label{eq:bound volume with g0}
 \big|B_{4r_k}^0(p_{j+1})\big|_{g_0}\leq C_0\cdot (4r_k)^n. 
 \end{equation}
Then by Proposition \ref{prop:length area}, we can take $D_{j+1}$ satisfying
\[B_{3r_k}^0(p_{j+1})\setminus \bigcup_{i=0}^jD_i\subset D_{j+1}\subset B_{4r_k}^0(p_{j+1})\setminus \bigcup_{i=0}^jD_i\]
and 
\begin{align} \label{eq:bound partial Dj}
\Big|\partial D_{j+1}\cap \Int (M\setminus \bigcup_{i=0}^jD_i)\Big| &\leq  \frac{1}{r_k}\cdot \Big| B_{4r_k}^0(p_{j+1})\setminus \bigcup_{i=0}^jD_i\Big|_{g_0}^{\frac{1}{n}}\cdot \Big| B_{4r_k}^0(p_{j+1})\setminus \bigcup_{i=0}^jD_i\Big|^{\frac{n-1}{n}}\\
&\leq 4C_0C(r_k)\cdot \Big| B_{r_k}^0(p_{j+1})\setminus \bigcup_{i=0}^jD_i\Big|^{\frac{n-1}{n}}. \nonumber
\end{align}
Here in the last inequality, we used \eqref{eq:bound 4r ball by r ball} and \eqref{eq:bound volume with g0}. Then there exists an integer $m\geq 0$ such that after $m$ many steps, we have $\{D_j\}_{j=1}^m$ such that for all $p\in M$,
\[ \Big|B_{r_k}^0(p)\setminus \bigcup_{j=0}^m D_j\Big|<\alpha_k.   \]
This gives that these domains $\{D_j\}$ satisfy the second item.

Now we are going to verify that these domains satisfy the first requirement. From the fact of $|D_j|_g>\alpha_k=1/k$, we conclude that $m\leq k-1$. Recall that $D_{j}\subset B^0_{4r_k}(p_j)$. Then we have
\[ |D_j|_{g_0}\leq |B_{4r_k}^0(p_{j})|_{g_0}\leq C_0(4r_k)^n<1.  \]
Moreover,
\begin{align*}
\Big|\bigcup_{j=1}^m\partial D_j\Big|&= \sum_{j=0}^{m-1}\Big|\partial D_{j+1}\cap \Int (M\setminus \bigcup_{i=0}^jD_i)\Big| \\
&\leq 4C_0C(r_k)\sum_{j=0}^{m-1}\Big | B_{r_k}^0(p_{j+1})\setminus \bigcup_{i=0}^jD_i\Big|^{\frac{n-1}{n}}\\
&\leq 4C_0C(r_k)\cdot m^{\frac{1}{n}}\cdot \Big(\sum_{j=0}^{m-1}\Big|B_{r_k}^0(p_{j+1})\Big|\Big)^{\frac{n-1}{n}}	\\
&\leq 4C_0C(r_k)|M|^{\frac{n-1}{n}}\cdot k^{\frac{1}{n}}\leq 4C_0C(1)|M|^{\frac{n-1}{n}}\cdot k^{\frac{1}{n}}.	
	\end{align*}
Here the first inequality is from \eqref{eq:bound partial Dj}; we used the H\"older's inequality in the second one; the third one follows from the fact of $B_{r_k}^0(p_i)\cap B_{r_k}^0(p_j)=\emptyset$ for $i\neq j$; for the last one, we used $r_k\leq 1$.
So far, Claim \ref{claim:thin-thick decomposition} is proved.
	\end{proof}%end of thin-thick decompositon

Denote by $D_{m+1}=\overline{M\setminus \cup_{j=1}^m D_j}$ and $k_j= k|D_j|/|M|$ for all $1\leq j\leq (m+1)$. Note that $|D_j|_{g_0}\leq 1$. Then by Theorem \ref{thm:cover of thick part} (using $k=[k_j]+1$ and $N=D_j$ there), for each $1\leq j\leq m$, there exists a finite cover $\{\oli U_i^j\}_{i}$ of $D_j$ such that 
\begin{equation}\label{eq:bounds for Dj}
 \big|\cup_i\partial U_i^j\cap (\mathrm{Int}D_j)\big|+ k_j\max_{i}\omega_1(U_i^j, g)\leq C_2|D_j|^{\frac{n-1}{n}}(1+[k_j])^{\frac{1}{n}}\leq 2C_2|D_j|^{\frac{n-1}{n}}k_j^{\frac{1}{n}} ,
\end{equation}
which also implies
\begin{equation}\label{eq:bounds for kUij}
k\max_i\omega_1(U_i^j, g)\leq \frac{k}{k_j}\cdot 2C_2\Big(\frac{k_j}{k}\cdot |M| \Big)^{\frac{n-1}{n}}k_j^{\frac{1}{n}}=2C_2|M|^{\frac{n-1}{n}}k^{\frac{1}{n}}.
\end{equation}
Note that $|B_{r_k}^0(p)\cap D_{m+1}|\leq \alpha_k$ for each $p\in D_{m+1}$. Applying Lemma \ref{lem:good decomposition of thin part} ($\alpha=\alpha_k$ and $r=r_k$), $D_{m+1}$ can be subdivided into disjoint open sets $\{V_j\}$ by $\cup_{j=1}^{L}\partial V_j$ satisfying the following:
 \begin{gather}\label{eq:bound partial Vj} 
 \Big|\bigcup_{j=1}^{L}\partial V_j\cap \mathrm{Int}D_{m+1}\Big|\leq (C_4/r_k)\cdot |D_{m+1}|^{\frac{1}{n}}_{g_0}\cdot |D_{m+1}|^{\frac{n-1}{n}};\\
\label{eq:bound 1-width Vj} \omega_1(V_j, g)\leq C_4\alpha_k^{\frac{n-1}{n}}\ \  \text{ for } \ \ 1\leq j\leq L.
\end{gather}
Here $C_4=5C_0(K+C(1/2))C(1)>C(r_k/2)C(r_k)+(4C_0+1)K\cdot C(r_k)$. 

\medskip
Note that $M$ is covered by $\{{\oli D}_j\}_{j=1}^{m+1}$. Hence $M$ is subdivided into $\cup_{j=1}^m\{U_i^j\}_i\cup \{V_l\}_{l=1}^L$. Then by Gromov \cite{Gro88} and Guth \cite{Guth09} (see also \cite{GL17}*{Proof of Theorem 7.1}),
\begin{align*}
  \omega_k(M,g)&\leq \sum_{j=1}^m \sum_i\big|\partial U_i^j\cap \mathrm{Int}D_j\big|+\Big|\bigcup_{j=1}^m\partial D_j\Big|+\Big|\bigcup_{j=1}^{L}\partial V_j\cap \mathrm{Int}D_{m+1}\Big|+k\max_{i,j}\omega_1(U_i^j,g)+\\
  &\ \ +k\max_{1\leq j\leq L}\omega_1(V_j,g)\\
  &\leq \sum_{j=1}^m2C_2|D_j|^{\frac{n-1}{n}}k_j^{\frac{1}{n}}+ 4C_0C(1)|M|^{\frac{n-1}{n}}k^{\frac{1}{n}} + (C_4/r_k)\cdot |D_{m+1}|^{\frac{1}{n}}_{g_0}\cdot |D_{m+1}|^{\frac{n-1}{n}}+\\
  &\  \ +2C_2|M|^{\frac{n-1}{n}}k^{\frac{1}{n}}+C_4\alpha_k^{\frac{n-1}{n}}\cdot k\\
  &\leq 2C_2\Big(\sum_{j=1}^m|D_j|\Big)^{\frac{n-1}{n}}\Big(\sum_{j=1}^m k_j\Big)^{\frac{1}{n}}+4C_0C(1)|M|^{\frac{n-1}{n}}k^{\frac{1}{n}}+8C_0C(1)C_4\cdot|M|^{\frac{n-1}{n}}k^{\frac{1}{n}}+\\
  &\ \ +(2C_2+C_4)|M|^{\frac{n-1}{n}}k^{\frac{1}{n}}\\
  &\leq \big(4C_2+13C_0C(1)C_4\big )|M|_{ g}^{\frac{n-1}{n}}k^{\frac{1}{n}}.
  \end{align*}
Here the second inequality is from \eqref{eq:bounds for Dj}, \ref{eq:bound partial Vj}, \eqref{eq:bounds for kUij} and \eqref{eq:bound 1-width Vj} and Claim \ref{claim:thin-thick decomposition}; in the third inequality, we used the H\"older's inequality for the first item, and the fact $|D_{m+1}|_{g_0}\leq |M|_{g_0}=|M|_{g}$ for the third item. Then we conclude that for any $k\geq \bar k$,
\begin{equation}\label{eq:for large k}
\omega_k(M,g)\leq C_3|M|^{\frac{n-1}{n}}k^{\frac{1}{n}},
\end{equation}
where $C_3:=4C_2+13C_0C(1)C_4$.

If $\bar k=1$, then we are done. Otherwise, it remains to estimate $\omega_k(M,g)$ for $k< \bar k$. Note that in this case,
\[   \bar k\leq 2\frac{|M|}{2C(1)}\leq |M|=|M|_{g_0}. \]
Then by \eqref{eq:for large k},
\[\omega_{\bar k}(M,g)\leq C_3|M|^{\frac{n-1}{n}}{\bar k}^{\frac{1}{n}}\leq  C_3|M|^{\frac{n-1}{n}}|M|_{g_0}^\frac{1}{n} . \]
Recall that $\omega_k(M,g)\leq \omega_{\bar k}(M,g)$ for $1\leq k\leq \bar k$. Thus we conclude that for all $k\geq 1$,
\[ \omega_k(M,g)\leq C_3|M|^{\frac{n-1}{n}}(k^{\frac{1}{n}}+|M|_{g_0}^\frac{1}{n}).\]

\end{proof}

\appendix

\section{Proof of Theorem \ref{thm:good cover for 1-width}}\label{sec:proof of GL-theorem}
	\begin{proof}[Proof of Theorem \ref{thm:good cover for 1-width}]
		We follows the steps given by Glynn-Adey and Liokumovich in \cite{GL17}, where they proved this theorem for $N=M$. Here we give the outline and point out some necessary modifications.
		
		Suppose that $N$ has smooth boundary. For any $\epsilon_0\in (0,1)$, take $\bar r(M,N,\epsilon_0)$ such that:
		\begin{itemize}
			\item for every $x\in \partial N$, we have that $B_r(x)$ is $(1 + \epsilon_0)$-bilipschitz diffeomorphic to the Euclidean ball of radius $r$ and $B_r(x)\cap N$ is mapped onto a half-ball under the difformorphism. Denote by $B_{r}^+(x)=B_r(x)\cap N$;
			\item the monotonicity formula \cite{GLZ16}*{Theorem 3.4} holds.
		\end{itemize}
		
		From now on, we fix some $\epsilon_0<1$.

		\medskip
		\noindent{\bf Step 1: }{\em Suppose that $N$ has smooth boundary. There exists $\epsilon=\epsilon(M,N,\bar r)$ satisfying the following:  for any domain $D\subset N$ with $|D|<\epsilon$, there exists a collection of domains $D(=:D_0)\supset D_1\supset D_2\supset \cdots\supset D_m$ satisfying
			\begin{itemize}
				\item $D_m\subset \Int N$;
				\item $|\partial D_j\cap \Int N|\geq |\partial D_{j+1}\cap \Int N|$ for $0\leq j\leq m-1$;
				\item for $0\leq j\leq m-1$, $D_j\setminus D_{j+1}$ is contained in some ball of radius $\bar r$ and center $x\in \partial N$;
		\end{itemize}}
		\begin{proof}[Proof of Step 1]
			Suppose that $x\in \partial D_j\cap \partial N$, now we construct $D_{j+1}\subset D_j$. By the co-area formula, we can find $r'\in(3\bar r/4,\bar r/4)$ such that $\partial D_j\cap \Int N$ is transverse to $\partial B_{r'}(x)$ and 
			\[ |D_j\cap \partial B_{r'}(x)|\leq (8/\bar r)\cdot |D_j\cap B_{r'}(x)|.  \]
			Denote by $S=\llbracket D_j\cap \partial B_{r'}(x)\rrbracket$. Let $T$ be the minimizing current $T$ among all $T'\in \mc Z_{n-1}( B^+_{r'}(x),\partial  B^+_{r'}(x);\mb Z_2)$ with $\spt(\partial T'-\partial S)\subset \partial N$. Then by the regularity theory \cite{Mor03}*{Theorem 4.7} (see also \cite{GLWZ19}*{Theorem 4.7}), $T$ is induced by a free boundary hypersurface $\Sigma$ with $(n-8)$-dimensional singular set. By taking $\epsilon$ small enough, from the monotonicity formula \cite{GLZ16}*{Theorem 3.4}, $\Sigma\cap \partial N\cap B_{\bar r/2}(x)=\emptyset$. Using the monotonicity formula again, $\Sigma\cap B_{\bar r/4}(x)=\emptyset$. Note that by the isoperimetric choice \cite{LZ16}, there exists $V\subset B_{\bar r}^+(x)$ such that $\partial \llbracket V\rrbracket=T-S$ and the volume of $V$ is small. Hence $V$ does not contain $B^+_{\bar r/4}(x)$. Together with the fact of $\partial V$ does not intersect $B^+_{\bar r/4}(x)$, we conclude that $V\cap B_{\bar r/4}^+(x)=\emptyset$. Now we define
			\[D_{j+1}=D_j\cap (N\setminus (B^+_{\bar r}(x)\setminus V) ).\]
			Clearly, $D_j\setminus D_{j+1}$ is contained in $B_{\bar r}^+(x)$. Note that $T$ is minimizing in $B_{\bar r}^+(x)$. Then it is minimizing in $B_{\bar r}^+(x)\setminus V$, i.e.
			\[ |\Sigma\cap D_j|\leq |\partial D_j\cap \Int ( B^+_{\bar r}(x)\setminus V)|.\]
			This implies
			\[ |\partial D_{j}\cap \Int N|-|\partial D_{j+1}\cap \Int N|=|\partial D_j\cap \Int ( B^+_{\bar r}(x)\setminus V)|- |\Sigma\cap D_j|\geq0. \]
			Thus Step 1 is completed.	
		\end{proof}%end of step 1
		
		\medskip
		\noindent{\bf Step 2: }{\em Suppose that $N$ has smooth boundary. There exist constants $\beta_1=\beta_1(n)$ and $\epsilon=\epsilon(M,N,\bar r)$ such that for any domain $D\subset N$ with $|D|\leq \epsilon$, the following bound holds:
			\begin{equation}
			\omega_1(D,g)\leq \beta_1|D|^{\frac{n-1}{n}}+|\partial D\cap \Int N|.
			\end{equation} }
		\begin{proof}[Proof of Step 2]	
			Let $\{ D_j\}_{j=1}^m$ be the domains constructed in Step 1. Then repeating the process inside $N$ (see also \cite{GL17}*{Proposition 4.3}), there exists $D_m\supset D_{m+1}\supset\cdots\supset D_L$ such that 
			\begin{itemize}
				\item $|\partial D_j\cap \Int N|\geq |\partial D_{j+1}\cap \Int N|$ for $m\leq j\leq L-1$;
				\item for $m\leq j\leq L$, $D_j\setminus D_{j+1}$ is contained in some ball of radius $\bar r$ and center $x\in N$, where $D_{L+1}=\emptyset$;
			\end{itemize}
			By \cite{Guth07}, there exists $\beta_1=\beta_1(n)$ such that for $0\leq j\leq L$,
			\begin{equation}\label{eq:bound Dj-Dj+1}
			\omega_1(D_j\setminus D_{j+1},g)\leq \beta_1|D_j\setminus D_{j+1}|^{\frac{n-1}{n}}.
			\end{equation}
			Now let $\Phi_j$ be a sweepout of $D_j\setminus D_{j+1}$ having no concentration of mass. Then there exist lifting maps $\wti \Phi_j:[0,1]\rightarrow \mc C(D_j\setminus D_{j+1})$ such that 
			\[ \partial \circ\wti\Phi_j=\Phi_j \ \ \  \text{ for } \ \ \ 0\leq j\leq L.  \]
			Without loss of generality, we assume that $\wti \Phi_j(0)=0$, $\wti\Phi_j(1)=\llbracket D_j\setminus D_{j+1}\rrbracket$. By \cite{GL17}*{Proposition 2.3}, we can construct a sweepout of $D$ as follows: we first define $\wti \Phi:[0,1]\rightarrow \mc C(D)$ by
			\[ \wti \Phi(t)=\wti \Phi_{L-j}\Big((L+1)(t-\frac{j}{L+1})\Big) +\llbracket D_{L+1-j}\rrbracket \ \ \text{ for } \ \ \frac{j}{L+1}\leq t\leq \frac{j+1}{L+1}. \]
			Then $\Phi=\partial \circ \Phi$ is the desired sweepout, which has no concentration of mass. Such a construction gives that 
			\[ \omega_1(D,g)\leq \max_{0\leq j\leq L}\{\omega_1(D_j\setminus D_{j+1},g)+|\partial D_j\setminus \partial D| \}  .\]
			Together with \eqref{eq:bound Dj-Dj+1}, we have
			\[   \omega_1(D,g)\leq  \beta_1|D|^{\frac{n-1}{n}}+|\partial D\cap N|.\]

		\end{proof}%end of step 2

		\medskip
		\noindent{\bf Step 3: }{\em Suppose that $N$ has smooth boundary. There exists $\beta_2=\beta_2(n)$ such that for any domain $D\subset N$, the following bound holds
			\begin{equation}\label{eq:bound omega1 for smooth D}
			\omega_1(D,g)\leq \beta_2\cdot (1+|D|_{g_0}^{\frac{1}{n}})|D|^{\frac{n-1}{n}}+2|\partial D\cap \Int N|.
			\end{equation}}
		\begin{proof}[Proof of Step 3]
			We use the argument in \cite{GL17}*{Theorem 5.1}. Let $\epsilon_1=25^{-n}\cdot\epsilon$. Take $\beta_2(n)=\beta_1(n)+3c(n)\cdot \Big[1-(1-25^{-n})^{\frac{n-1}{n}}\Big]$. Here $c(n)$ is the constant in \cite{GL17}*{Lemma 3.4}. It follows that 
			\begin{equation}\label{eq:beta2 and cn}
			\Big[1-(1-25^{-n})^{\frac{n-1}{n}}\Big]\beta_2(n)\geq 3c(n).
			\end{equation}
			By Step 2, for $k\leq 25^n$, \eqref{eq:bound omega1 for smooth D} holds for $D$ with $|D|\leq k\epsilon_1$. We proceed by
			induction on $k$.
			
			Suppose the inequality holds for compact domains with volume at most $k\epsilon$. Then for any $D\subset N$ with $k\epsilon_1<|D|\leq (k+1)\epsilon_1$. By Theorem \ref{thm:isoperi by GL}, there exists a hypersurface $\Sigma$ subdividing $D$ into $D_0$ and $D_1$ such that $|D_j|\leq (1-25^{-n})|D|$ (for $j=0,1$) and  
			\begin{equation}\label{eq:Sigma bound by D}
			|\Sigma|\leq c(n)|D|^{\frac{n-1}{n}}(1+|D|_{g_0}^{\frac{1}{n}}).
			\end{equation}
			Then using the construction of sweepouts in Step 2, we have
			\begin{equation}\label{eq:bound omega1 inductively}
			\omega_1(D,g)\leq \max_{j\in\{0,1\}}\{ \omega_1(D_j,g)+|\partial D_j\setminus \partial D| \}.
			\end{equation}
			Note that for $j=0,1$,
			\[|D_j|\leq (1-25^{-n})|D|\leq |D|-\epsilon_1<(k+1)\epsilon_1-\epsilon_1<k\epsilon _1.\]
			Hence by the assumption, 
			\begin{align*}
			\omega_1(D_j,g)&\leq \beta_2\cdot (1+|D_j|_{g_0}^{\frac{1}{n}})|D_j|^{\frac{n-1}{n}}+2|\partial D_j\cap \Int N|\\
			&\leq  \beta_2\cdot (1+|D|_{g_0}^{\frac{1}{n}})|D|^{\frac{n-1}{n}}\cdot(1-25^{-n})^{\frac{n-1}{n}}+2|\partial D\cap \Int N|+2|\Sigma|\\
			&\leq (\beta_2-3c(n))(1+|D|_{g_0}^{\frac{1}{n}})|D|^{\frac{n-1}{n}}+2|\partial D\cap \Int N|+2|\Sigma|\\
			&\leq \beta_2\cdot (1+|D|_{g_0}^{\frac{1}{n}})|D|^{\frac{n-1}{n}}+2|\partial D\cap \Int N|-|\Sigma|.
			\end{align*}
			Here the third inequality is from \eqref{eq:beta2 and cn} and we used \eqref{eq:Sigma bound by D} in the last one. Then together with \eqref{eq:bound omega1 inductively}, we conclude that 
			\[ \omega_1(D,g)\leq  \beta_2\cdot (1+|D|_{g_0}^{\frac{1}{n}})|D|^{\frac{n-1}{n}}+2|\partial D\cap \Int N|. \]
			This finishes Step 3.
		\end{proof}   %end of Step 3

		\medskip
		\noindent{\bf Step 4: }{\em We prove the theorem for general compact domain $N$ (having piecewise smooth boundary).}
		\begin{proof}[Proof of Step 4]
			Now let $N$ be a compact domain with piecewise smooth boundary. Then we have a tubular neighborhood $U$ of $N$ such that $U$ has smooth boundary and $|U|_{g_0}\leq 2|N|_{g_0}$ and $|U|\leq 2|N|$. Then by Step 3,
			\[\omega_1(U,g)\leq \beta_2\cdot (1+|U|_{g_0}^{\frac{1}{n}})|U|^{\frac{n-1}{n}}\leq 2\beta_2\cdot (1+|N|_{g_0}^{\frac{1}{n}})|N|^{\frac{n-1}{n}} .\]
			Then the desired inequality follows from 
			\[ \omega_1(N,g)\leq \omega_1(U,g)\]
			if we take $K=2\beta_2(n)$.
		\end{proof}   %end of Step 4
		
		So far, Theorem \ref{thm:good cover for 1-width} is proved.
	\end{proof}

\bibliographystyle{amsalpha}
\bibliography{minmax}
\end{document}